\documentclass{amsart}
\usepackage{graphicx}
\usepackage{amssymb}

\newtheorem{thm}{Theorem}[section]

\newtheorem{lem}[thm]{Lemma}
\newtheorem{prop}[thm]{Proposition}
\theoremstyle{definition}

\theoremstyle{remark}
\newtheorem{rem}[thm]{Remark}
\theoremstyle{conjecture}

\theoremstyle{example}

\begin{document}

\title[Non-convex planar domains]{Sobolev regularity for optimal transport maps of non-convex planar domains}

\author{Connor Mooney}
\address{Department of Mathematics, UC Irvine}
\email{\tt mooneycr@math.uci.edu}
\author{Arghya Rakshit}
\address{Department of Mathematics, UC Irvine}
\email{\tt arakshit@uci.edu}

% ----------------------------------------------------------------
\begin{abstract}
We prove a sharp global $W^{2,\,p}$ estimate for potentials of optimal transport maps that take a certain class of non-convex planar domains to convex ones.
\end{abstract}
\maketitle

% ----------------------------------------------------------------

\section{Introduction}
Optimal transport maps play an important role in physics, geometry, economics, and meteorology. The regularity of optimal transport maps is a delicate matter that for the most part has focused on the case that the source and target domains are convex. However, this condition is not satisfied in many applications. In this paper we initiate the study of the Sobolev regularity of optimal transport maps in the plane, where the source domain is non-convex.

The setting is as follows. Let $\Omega_1$ and $\Omega_2$ be bounded domains in $\mathbb{R}^2$ of unit area. We assume that $\Omega_2$ is convex. Then the optimal transport map from $\Omega_1$ to $\Omega_2$ is the gradient map of a convex function $u$ on $\mathbb{R}^2$ which satisfies (see \cite{C5}):
\begin{equation}\label{MA}
\det D^2u = \chi_{\Omega_1}
\end{equation}
in the Alexandrov sense, $u$ is smooth and locally uniformly convex in $\Omega_1$, and
\begin{equation}\label{BoundaryCondition}
\nabla u(\Omega_1) = \Omega_2.
\end{equation}
We assume further that $\Omega_1$ is a convex domain $\Omega_0$ with a finite number of disjoint, $C^{1,\,1}$, uniformly convex holes a positive distance $\delta$ from $\partial \Omega_0$ and from each other removed. Our main result is:
\begin{thm}\label{Main}
We have $u \in C^{1,\,1/2}\left(\overline{\Omega_1}\right)$ with norm depending only on the diameters of $\Omega_1$ and $\Omega_2$, $\delta$, and the lower and upper bounds for the boundary curvatures of the holes in $\Omega_1$. We also have $u \in W^{2,\,p}(\Omega_1)$ for any $p < 2$, with norm depending only on the same quantities and $p$.
\end{thm}

Theorem \ref{Main} is sharp. To see this, consider the radially symmetric example where $\Omega_1$ is an annulus with inner radius $r$, $\Omega_2$ is a disk, and the potential is
\begin{equation}\label{ModelExample}
u(x) = \int_0^{|x|} (s^2-r^2)^{1/2}_+\,ds.
\end{equation}
Below we will refer to (\ref{ModelExample}) as the model example.

One motivation for Theorem \ref{Main} comes from the semigeostrophic equations (SGEs) from meteorology. The SGEs lead one to consider optimal transport maps that take a bounded density on the torus to the uniform one (see e.g. \cite{F}). When the source density is bounded between positive constants, $W^{2,\,1}$ estimates for the potential are available (\cite{DF}, \cite{DFS}, \cite{Sch}), which lead to long-time existence results for the SGEs (\cite{ACDF1}, \cite{ACDF2}, \cite{F}). However, in physically interesting cases, the SGEs involve optimal transport maps where the source density is allowed to vanish. In this case, $W^{2,\,1}$ estimates for the potential do not always hold (see \cite{M}). An important special situation is when the source density is the characteristic function of a domain (which need not be convex), as in the situation of Theorem \ref{Main}. For the SGEs, this corresponds to a fully nonlinear analogue of the vortex patch problem for the 2D Euler equations. Global $W^{2,\,1}$ estimates for optimal transport maps of non-convex domains may be useful for extending long-time existence results for the SGEs to this situation.

More generally, theorem \ref{Main} can be viewed as a step towards obtaining global regularity results for optimal transport maps of general non-convex source domains. The global regularity of optimal transport maps in the case of convex source and target domains is well-studied. Caffarelli proved that, in this case, the potentials are $C^{1,\alpha}$ up to the boundary, and $C^{2,\alpha}$ up to the boundary provided the domains are $C^2$ and uniformly convex \cite{C4}. Here $\alpha$ is small. The conditions on the domains required for global $C^{2,\alpha}$ regularity of the potential were recently relaxed to $C^{1,1}$ and merely convex \cite{CLW1}, and even slightly non-convex but close to convex in the $C^{1,1}$ sense \cite{CLW2}. In two dimensions, Savin and Yu showed that convexity of the domains is enough to get global $W^{2,\,p}$ regularity for any $p < \infty$ \cite{SY}. 
As for the case of non-convex source domain, in \cite{AC} the authors obtain global $C^{1,\alpha}$ estimates for potentials of optimal transport maps in any dimension when the densities are bounded between positive constants, the target domain is convex, and the source domain is a convex set with finitely many convex holes removed, using ideas from \cite{C1}. (Again, here $\alpha$ is small). Our methods (described below) are quite different from those in \cite{AC}, and the smoothness of the densities and the regularity properties of the holes play a delicate role in our analysis. We remark that our methods in fact apply near any ``uniformly concave" part of the boundary of a general smooth planar source domain.

Our strategy is as follows. First, we may focus our attention on a neighborhood of the holes in $\Omega_1$ (the ``concave part" of the boundary of $\Omega_1$), thanks to work of Savin-Yu which shows the $W^{2,\,p}$ regularity of $u$ (for any $p$) near the ``convex part" $\partial \Omega_0$ of the boundary of $\Omega_1$ \cite{SY}. We carefully analyze the geometry of the sections of $u$ (defined in Section \ref{Preliminaries}) which are centered at concave boundary points. We show that there are three possible cases, all of which are ``good." The first case is that the complement of $\Omega_1$ fills only a tiny fraction of the section. In this case, morally speaking $u$ solves $\det D^2u = 1$ in the whole section and we can control section geometry at smaller scales using the regularity theory for the Monge-Amp\`{e}re equation. The second case is that the complement of $\Omega_1$ fills a positive universal fraction of the section and the long axis of the section is transversal to the boundary. In this case we show that renormalization by an affine transformation flattens the boundary, and we are again in a good situation where section geometry can be controlled at smaller scales using Pogorelov-type estimates for the Monge-Amp\`{e}re equation. The last case is that the complement of $\Omega_1$ fills a positive universal fraction of the section, and the long axis of the section is roughly tangent to the boundary. In this case we are in a situation that resembles what happens at every inner boundary point for the model example, which has the desired regularity properties. Our analysis near the holes is valid for any solution of the Monge-Amp\`{e}re equation, and does not use the convexity of $\Omega_2$.

In the course of the proof we also prove a new interior second derivative estimate for solutions to $\det D^2w = \chi_{\{x_2 > 0\}}$, which is special to two dimensions. The classical Pogorelov estimate bounds the tangential second derivative $w_{11}$. Although this suffices for our application, using the partial Legendre transform we are also able to bound the ratio $|w_{12}|/w_{11}$ from above (see Proposition \ref{Pogorelov2}). As a result, the sections of $w$ centered on $\{x_2 = 0\}$ are well-approximated by ellipsoids whose axes are aligned with the coordinate axes. This result simplifies our proof, and may be useful for future applications.

The paper is organized as follows. In Section \ref{Preliminaries} we discuss some preliminary results about the geometry of centered sections, as well as some Pogorelov-type estimates (including the one mentioned in the previous paragraph). In Section \ref{KeyLemmas} we prove several key lemmas, corresponding to the three scenarios mentioned above. In Section \ref{ProofMain} we prove Theorem \ref{Main}. In Section \ref{FutureDirections} we discuss some future directions. Finally, in the appendix Section \ref{Appendix} we prove some of the preliminary results.

\section*{Acknowledgments}
The authors gratefully acknowledge the support of NSF CAREER grant DMS-2143668, an Alfred P. Sloan Research Fellowship, and a UC Irvine Chancellor's Fellowship. C. Mooney would like to thank A. Figalli for discussions on a related problem which led to Proposition \ref{Pogorelov2}.

%%%%%%%%%%%%%%

\section{Preliminaries}\label{Preliminaries}

For the remainder of the paper we fix a constant $\delta > 0$. We let $\mathcal{F}_{\delta}$ denote the space of convex functions on $\mathbb{R}^2$ that satisfy (\ref{MA}) and (\ref{BoundaryCondition}), where $\Omega_1$ and $\Omega_2$ have unit area and are contained in  $B_{\delta^{-1}}$, $\Omega_2$ is convex, and the source domain $\Omega_1$ consists of a convex domain $\Omega_0$ with convex holes removed, where the holes are separated a distance at least $\delta$ from one another and from the boundary of $\Omega_0$, and the boundaries of the holes have lower and upper bounds $\delta,\,\delta^{-1}$ on their curvature. (Here $\Omega_i$ are not fixed, they are any domains satisfying the above conditions). We note that $\mathcal{F}_{\delta}$ is a compact family, namely, any sequence in $\mathcal{F}_{\delta}$ contains (after possibly adding constants) a subsequence that converges locally uniformly on $\mathbb{R}^2$ to a function in $\mathcal{F}_{\delta}$. The local uniform convergence follows from the fact that the gradients lie in $B_{\delta^{-1}}$ and the Arzela-Ascoli theorem. The fact that the limit lies in $\mathcal{F}_{\delta}$ uses the weak convergence of Monge-Amp\`{e}re measures under local uniform convergence \cite{Gut}. We call constants depending only on $\delta$ universal, and we say that positive quantities $a$ and $b$ satisfy $a \sim b$ if their ratio is trapped between positive universal constants. Below $c$ will denote a small positive universal constant which may change from line to line.

Let $u \in \mathcal{F}_{\delta}$. For any $x \in \overline\Omega_1$ and $h > 0$ there exists an affine function $L_{x,h}$ such that 
$$L_{x,h}(x) = u(x) + h$$ 
and such that the set $\{u < L_{x,h}\}$ is bounded and has center of mass $x$ (see \cite{C4}). We call $\{u < L_{x,h}\}$ the centered section of height $h$ at $x$, and we denote it by $S_h^u(x)$. 
One can show that $u$ is not linear when restricted to any line segment centered at a point in $\overline{\Omega}_1$ (see Appendix). Combined with a compactness argument, this shows that there exists a universal modulus of continuity $\omega$ such that for any $x \in \overline{\Omega_1}$ and any $h < 1$,
\begin{equation}\label{DiamBound}
\text{diam}(S_h^u(x)) \leq \omega(h).
\end{equation}
In particular, for $h < c_0$ universal, we have that $S_h^u(x)$ intersects at most one connected component of $\Omega_1^c$ for any $x \in \overline{\Omega_1}$. Below we will always assume that $h \in (0,\,c_0)$, and we will only consider sections centered in $\overline{\Omega_1}$ that are either contained in $\Omega_1$ or intersect a hole in $\Omega_1$.

Such centered sections satisfy the area estimate
\begin{equation}\label{Area}
|S_h^u(x)| \sim h.
\end{equation}
This estimate follows from the universal positive density of $\Omega_1$ in such sections. By a version of John's Lemma, there exist rectangles $R_h(x)$ centered at $0$ of area $4h$ such that
\begin{equation}\label{SectionTrapping}
x + cR_h(x) \subset S_h^u(x) \subset x + c^{-1}R_h(x).
\end{equation}
We denote the short and long side lengths of $R_h(x)$ by $2\lambda_h(x)$ and $2\Lambda_h(x),$ respectively, and we define the eccentricity of $R_h(x)$ by the quantity
$$\eta_h(x) = \frac{\Lambda_h(x)}{\lambda_h(x)}.$$

Finally, we have the following engulfing property (see Appendix), which allows us to compare sections in $\Omega_1$ tangent to a hole with a section centered on the boundary of the hole:
\begin{prop}\label{Engulfing}
If $u \in \mathcal{F}_{\delta},\, y \in \overline{S_h^u(x)} \cap (\partial\Omega_1 \backslash \partial \Omega_0)$ and $S_h^u(x) \subset \Omega_1$, then 
$$S_h^u(x) \subset y + R_{c^{-1}h}(y).$$
\end{prop}

We now state some Pogorelov-type estimates. We say that a convex domain $\Omega$ is normalized if $B_c \subset \Omega \subset B_{c^{-1}}$ for $c > 0$ universal. The first result is Pogorelov's interior $C^2$ estimate (see e.g. \cite{Gut}):

\begin{prop}\label{Pogorelov1}
If $\det D^2w = 1$ in $S_1^w(0)$ and $S_1^w(0)$ is normalized, then $|D^2w| < c^{-1}$ in $\frac{1}{2}S_{1}^w(0)$ and
$$B_{ch^{1/2}} \subset S_h^w(0) \subset B_{c^{-1}h^{1/2}}$$
for all $h < 1$ and some $c > 0$ universal.
\end{prop}

Combining Proposition \ref{Pogorelov1} with the affine invariance of the Monge-Amp\`{e}re equation and the area estimate (\ref{Area}) we have
\begin{equation}\label{HessianBound}
|D^2u(x)| \sim \eta_h(x)
\end{equation}
whenever $S_h^u(x) \subset \Omega_1$.

The next estimate is a variant of Pogorelov's interior $C^2$ estimate with flat boundary, which to our knowledge is new:

\begin{prop}\label{Pogorelov2} 
If $\det D^2w = \chi_{\{x_2 > 0\}}$ in $S_1^w(0)$ and $S_1^w(0)$ is normalized, then
$$\sup_{\frac{1}{2}S_1^w(0) \cap \{x_2 > 0\}} \left(w_{11} + \frac{|w_{12}|}{w_{11}}\right) \leq c^{-1}$$
for some universal $c > 0$.
\end{prop}

\noindent The upper bound on $w_{11}$ is the classical Pogorelov estimate (see \cite{C4}) and doesn't use that we are working in the plane. The upper bound on $|w_{12}|/w_{11}$ uses that we are in the plane.

\begin{proof}
We may assume after subtracting a linear function that $w|_{\partial S_1^w(0)} = 0$.
Let $w^*$ denote the partial Legendre transform of $w$ (see e.g. \cite{DS} for the definition and properties), which is convex in the first variable, concave in the second, and formally solves
$$\chi_{\{x_2 > 0\}}w^*_{11} + w^*_{22} = 0$$
in $B_{c}(0)$ for $c > 0$ universal. More precisely, $w^*$ is harmonic in $\{x_2 > 0\},$ linear on vertical segments in $\{x_2 < 0\},$ and moreover $w^*_2$ has the same limit from above and below on $\{x_2 = 0\}$ along vertical lines. It is not hard to verify the first two properties by approximating $\chi_{\{x_2 > 0\}}$ with smooth positive functions of $x_2$. The third property can be verified using that $w \in C^1$ (\cite{A}, \cite{FL}). Since $w^*$ is linear on vertical segments in $\{x_2 \leq 0\}$ we have
$$w^*_2(x_1,\,0) = a^{-1}(w^*(x_1,\,0) - w^*(x_1,\,-a))$$
for any $a > 0$. Choosing $a \sim 1$ and using that $w^*$ is convex (hence locally Lipschitz with locally universally bounded Lipschitz constant) in the horizontal directions, we conclude that $w^*_2$ is Lipschitz on $\{x_2 = 0\}$. In particular, $w^*_2$ is harmonic in $B_c \cap \{x_2 > 0\}$ and Lipschitz on $\{x_2 = 0\}$ (with locally universally bounded Lipschitz constant). It follows from harmonic function theory that $|w^*_{12}| < c^{-1}$ in $B_{c/2} \cap \{x_2 > 0\}$. Using the relation
$$w_{12} = -w^*_{12}w_{11}$$
we obtain the desired estimate on $|w_{12}|/w_{11}$.
\end{proof}

As a result of Proposition \ref{Pogorelov2}, under the same assumptions we can say that $S_h^w(0)$ is approximated by (contains and is contained in dilations by universal constants of) a rectangle with axes that are aligned with the coordinate axes for all $h < 1$. Indeed, if not, then $S_h^w(0)$ is approximated by an ellipsoid of the form $A_hB_1$, where
$$A_h(x_1,\,x_2) = \left(A\sqrt{h}(x_1 + Kx_2),\, A^{-1}\sqrt{h}x_2\right),$$
$A \geq c$ (this follows from the upper bound on $w_{11}$), and $|K| >> 1$. The function
$$v(x) = \frac{1}{h}w(A_hx)$$
satisfies the conditions of Proposition \ref{Pogorelov2}, and
$$w_{12} = v_{12} - Kv_{11}.$$
In $B_c(c e_2)$ we can find points where $v_{11} \sim 1$ and $|v_{12}| < c^{-1}$ to arrive at a contradiction of Proposition \ref{Pogorelov2} when $|K|$ is sufficiently large.

%%%%%%%%%%%%%%
\section{Key Lemmas}\label{KeyLemmas}
In this section we prove some lemmas that control the geometry of sections centered at concave boundary points in various scenarios. We will use several times below the standard fact that if $w_k$ are convex functions with $S_1^{w_k}(0)$ normalized, $w_k|_{\partial S_1^{w_k}(0)} = 0$ and $\det D^2w_k$ are uniformly bounded above, then a subsequence of $w_k$ converges uniformly to a convex function $w$ satisfying the same properties, and the Monge-Amp\`{e}re measures converge weakly to that of the limit.

The first lemma deals with the case that $\Omega_1^c$ bites only a small fraction of the section:
\begin{lem}\label{SmallFrac}
For all $M > 1$, there exists $\epsilon > 0$ depending on $\delta$ and $M$ such that if $x \in \partial \Omega_1 \backslash \partial \Omega_0$ and
$$\frac{|(x + R_h(x)) \cap \Omega_1^c|}{|R_h(x)|} \leq \epsilon,$$
then
$$\eta_{h/M}(x) \leq c_1^{-1}\eta_h(x)$$
and
$$R_{h/M}(x) \subset c_1^{-1}M^{-1/2}R_{h}(x).$$
Here $c_1 > 0$ is a universal constant.
\end{lem}
\begin{proof}
Assume by way of contradiction that the lemma is false. Then there is a sequence $u_k \in \mathcal{F}_{\delta}$, points $x_k$ on the boundaries of the holes in the source domains, and $h_k > 0$ such that the area fraction of the complements of the source domains in $x_k + R_{h_k}(x_k)$ tends to zero but the conclusions don't hold for $c_1$ to be determined. After performing a rigid motion we may assume that $x_k = 0$ and that that $R_{h_k}(0)$ have short side vertical. Then up to adding affine functions and taking a subsequence, the rescalings $w_k := h_k^{-1}u_k(\Lambda_{h_k}x_1,\,\lambda_{h_k}x_2)$ converge locally uniformly to a function $w$ that satisfies the conditions of Proposition \ref{Pogorelov1}. Applying the proposition we conclude that $B_{cM^{-1/2}} \subset S_{1/M}^{w_k}(0) \subset B_{c^{-1}M^{-1/2}}$ for $k$ large and $c$ universal. Scaling back we reach a contradiction provided $c_1 > 0$ was chosen small universal.
\end{proof}

\begin{rem}\label{RectangleComparison}
Here and below we use that if a rectangle $R_1$ centered at $0$ is approximated by (contains and is contained in dilations by universal constants times) another rectangle $R_2$ centered at $0$, then their side lengths $\lambda_i,\,\Lambda_i\, (i = 1,\,2)$ satisfy $\lambda_1 \sim \lambda_2,\, \Lambda_1 \sim \Lambda_2$. 
\end{rem}

We define
\begin{equation}\label{M1}
M_1 = c_1^{-6},
\end{equation}  
where $c_1$ is the universal constant in Lemma \ref{SmallFrac}, and we let $\epsilon_1$ be the corresponding volume fraction from the lemma.                

The next lemma deals with the case that $\Omega_1^c$ bites a positive fraction of the section, and the long axis is transversal to the boundary. We let $l_h(x)$ denote the length of the intersection of tangent line to $\partial \Omega_1$ at $x \in \partial \Omega_1$ with $x + R_h(x)$. This lemma uses the regularity and uniform convexity of the boundary.
\begin{lem}\label{SmallFrac2}
For all $M > 1$, there exists $\epsilon > 0$ depending on $\delta$ and $M$ such that if $x \in \partial \Omega_1 \backslash \partial \Omega_0$, and in addition
$$\eta_h(x) > c_2^{-1}M, \quad \frac{|(x + R_h(x)) \cap \Omega_1^c|}{|R_h(x)|} > \epsilon_1, \text{ and } \quad \frac{l_h(x)}{\Lambda_h(x)} \leq \epsilon,$$
then
$$\eta_{h/M}(x) \leq c_2^{-1}\eta_h(x).$$
Furthermore, if
$$\eta_{h/M}(x) > \eta_h(x),$$
then
$$R_{h/M}(x) \subset c_2^{-1}M^{-1/2}R_{h}(x).$$
Here $c_2 > 0$ is a universal constant.
\end{lem}
\begin{proof}
Assume by way of contradiction that the lemma is false. Then there is a sequence $u_k \in \mathcal{F}_{\delta}$, points $x_k$ on the boundaries of the holes in the source domains $\Omega_{1k}$, and $h_k > 0$ such that the first two inequalities hold for $R_{h_k}(x_k)$ and 
\begin{equation}\label{RatioContra}
l_{h_k}(x_k)/\Lambda_{h_k}(x_k) < 1/k,
\end{equation} 
but the conclusion is false for all $k$. After performing a rigid motion we may assume that $x_k = 0$, that $R_{h_k}(0)$ have short side vertical, and that $te_2$ and $se_1$ are contained in the source domain for $t \in (0,\,\lambda_{h_k}(0))$ and $s \in (-\Lambda_{h_k}(0),\,0)$. From hereon out we write $R_k = R_{h_k}(0)$, $\lambda_k = \lambda_{h_k}(0)$, and $\Lambda_k = \Lambda_{h_k}(0)$.

Let $s_k \geq 0$ be the slope of the tangent line to the boundary of the source domain at $0$ (we allow $s_k = \infty$). Inequality (\ref{RatioContra}) implies that 
\begin{equation}\label{RatioContra2}
s_k\Lambda_k\lambda_k^{-1} > k.
\end{equation}

We claim that Proposition \ref{Pogorelov2} applies to the limit of the sequence of rescalings $h_k^{-1} u_k(\Lambda_kx_1, \lambda_kx_2)$ (up to adding affine functions), after swapping $x_1$ and $x_2$ (see Figure \ref{Fig1}). If $s_k \geq 1$ we can conclude this by arguing as in the proof of Proposition $4.2$ in \cite{SY}, so assume otherwise. Uniform concavity of the boundary implies that $\{x_2 < s_kx_1 - cx_1^2\}$ contain the holes with $0$ in their boundary. Since 
$$R_k \cap \{x_2 < s_kx_1 - cx_1^2\} \subset R_k \cap \{|x_1| < c^{-1}(s_k^2 + \lambda_k)^{1/2}\},$$ 
the lower bound on the volume of the complement implies that
$$\Lambda_k^2 \leq c^{-1}(s_k^2 + \lambda_k).$$
We claim that $s_k^2 \geq \lambda_k$ for $k$ large. Indeed, if not, then the previous inequality implies that $\Lambda_k^2 \leq c^{-1}\lambda_k$, which combined with (\ref{RatioContra2}) gives $k < c^{-1}s_k\lambda_k^{-1/2},$ and since we assumed that $s_k\lambda_k^{-1/2} < 1$ we get a contradiction for $k$ large. We conclude that
\begin{equation}\label{LambdaBound}
\Lambda_k \leq c^{-1}s_k.
\end{equation}

By $C^{1,\,1}$ regularity of the holes, the complements of the domains contain $\{x_2 < s_kx_1 - c^{-1}x_1^2\} \cap R_k$. In new coordinates $\tilde{x}_1$ and $\tilde{x}_2$ where $x_1 = \Lambda_k\tilde{x}_1,\, x_2 = \lambda_k \tilde{x}_2$, the parabolic domains $\{x_2 < s_kx_1 - c^{-1}x_1^2\}$ become
$$\{\tilde{x}_2 < s_k\Lambda_k\lambda_k^{-1}\tilde{x}_1 - c^{-1}\Lambda_k^2\lambda_k^{-1}\tilde{x}_1^2\}.$$
Using the bound (\ref{LambdaBound}) on $\Lambda_k$, we see that these domains contain
$$\{\tilde{x}_2 < s_k\Lambda_k\lambda_k^{-1}(\tilde{x}_1 - c^{-1}\tilde{x}_1^2)\}.$$
By (\ref{RatioContra2}), the latter domains converge to the slab $\{0 < \tilde{x}_1 < c\}$ as $k$ tends to infinity.

\begin{figure}
 \begin{center}
    \includegraphics[scale=0.5, trim={0mm 17mm 0mm 0mm}, clip]{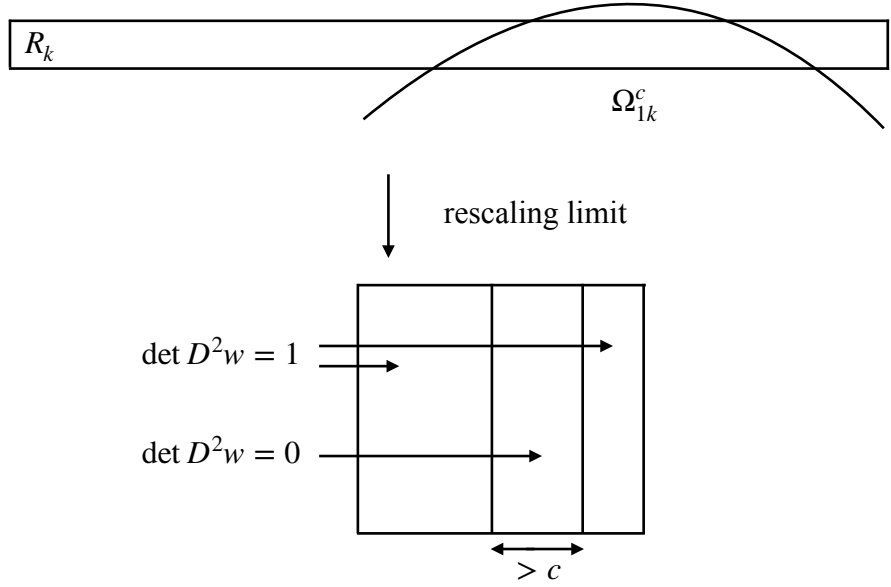}
\caption{Rescaling limit in the case of nontrivial exterior area and transversal boundary}
\label{Fig1}
\end{center}
\end{figure}

We conclude that a subsequence of $h_k^{-1} u_k(\Lambda_kx_1, \lambda_kx_2) + \text{ affine}_k$ converges to a limit function $w$ which satisfies that $\det D^2w = 1$ in $\{x_1 < 0\} \cap S_1^w(0)$, that $\det D^2w = 0$ in $\{0 < x_1 < c\} \cap S_1^w(0)$, and that $S_1^w(0)$ is normalized. A small modification of Proposition \ref{Pogorelov2} implies that $S_{1/M}^w(0)$ is approximated by a rectangle with axes aligned with the coordinate axes. Moreover, the upper bound on the vertical second derivative implies that the horizontal length $l$ and vertical length $L$ of this rectangle satisfy $l \leq c^{-1}L$. For $k$ large we conclude that
$$\eta_{h_k/M} \sim \eta_{h_k}l/L,$$
provided $\eta_{h_k}l/L > 1$. Since $L/l \leq c^{-1}M$ by the Lipschitz regularity of $w$ and the volume estimate for centered sections, the first inequality we assumed about $\eta_{h_k}$ guarantees this is satisfied. Thus, the eccentricity $\eta_{h_k/M}$ increased by at most a universal factor compared to $\eta_{h_k}$, and if $l < cL$ then we have $\eta_{h_k/M} < \eta_{h_k}$. If not, then we have that $S_{1/M}^w(0)$ is approximated by $B_{M^{-1/2}}$, thus for $k$ large the rectangles $R_{h_k/M}(0)$ are contained in universal dilations of $M^{-1/2}R_k$. This gives the desired contradiction.
\end{proof}

We now define
\begin{equation}\label{M2}
M_2 = c_2^{-6},
\end{equation}
where $c_2$ is the universal constant from Lemma \ref{SmallFrac2}, and we let $\epsilon_2 > 0$ be the corresponding length ratio from that lemma.

Finally, the remaining lemma is purely geometric, and also uses the regularity and convexity properties of the holes in $\Omega_1$. Below, $d$ denotes the distance function from $\Omega_1^c$.

\begin{lem}\label{ModelGeometry}
Assume that $x \in \partial \Omega_1 \backslash \partial \Omega_0$ and that 
$$\frac{|(x +R_h(x)) \cap \Omega_1^c|}{|R_h(x)|} > \epsilon_1, \quad \frac{l_h(x)}{\Lambda_h(x)} > \epsilon_2.$$ 
Then
$$\Lambda_h^2(x) + \sup_{x + R_h(x)} d \leq c_3^{-1}\lambda_h(x),$$
where $c_3 > 0$ is universal.
\end{lem}
\begin{proof}
Perform a rigid motion as in the proof of Lemma \ref{SmallFrac2} so that $x = 0$, the short side of $R_h(0)$ is vertical, and the hole lies beneath its tangent line at $0$ which has slope $s \geq 0$. Below we we will denote $R_h(0)$ by $R_h$, and we will similarly drop the notation $(0)$ from the other relevant quantities. We can assume that $\lambda_h < c\Lambda_h$, otherwise the lemma is obvious. The second inequality in the hypothesis implies that
\begin{equation}\label{sbound}
s \leq c^{-1}\lambda_h/\Lambda_h.
\end{equation}
The uniform concavity of $\partial \Omega_1$ implies that
$$R_h \cap \Omega_1^c \subset R_h \cap \{x_2 < sx_1 - cx_1^2\} \subset R_h \cap \{|x_1| < c^{-1}(s^2 + \lambda_h)^{1/2}\}.$$
Using the first inequality in the hypothesis we conclude that
$$\Lambda_h^2  < c^{-1}(s^2 + \lambda_h).$$
Using (\ref{sbound}) in the previous inequality gives
$$\Lambda_h^2 \leq c^{-1}(\lambda_h^2\Lambda_h^{-2} + \lambda_h),$$
and it follows that
\begin{equation}\label{lbound}
\Lambda_h^2 \leq c^{-1}\lambda_h.
\end{equation}
Furthermore, the $C^{1,\,1}$ regularity of $\partial \Omega_1$ implies that the hole has a boundary portion that lies above $\{x_2 = sx_1 - c^{-1}x_1^2\}$ in $\{|x_1| < \Lambda_h\}$. The distance of points in $R_h$ from $\Omega_1^c$ is thus bounded above by
$$c^{-1}(\lambda_h + |s|\Lambda_h + c^{-1}\Lambda_h^2).$$
Using (\ref{sbound}) and (\ref{lbound}) we arrive at the desired estimate.
\end{proof}

See Figure \ref{Fig2} for a summary of the results from Lemma \ref{ModelGeometry}.

\begin{figure}
 \begin{center}
    \includegraphics[scale=0.5, trim={10mm 50mm 0mm 20mm}, clip]{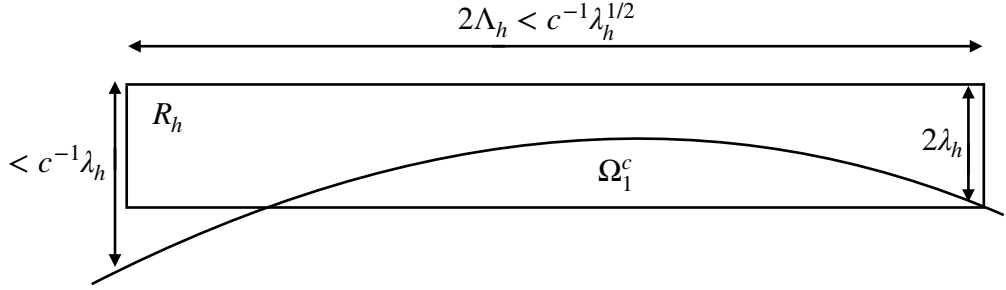}
\caption{``Model example geometry" in the case of nontrivial exterior area and roughly tangential boundary}
\label{Fig2}
\end{center}
\end{figure}

%%%%%%%%%%%%%%%%%%

\section{Proof of Theorem \ref{Main}}\label{ProofMain}
\begin{proof}[{\bf Proof of Theorem \ref{Main}}]

For each $x \in \Omega_1$, let $S_{\bar{h}(x)}^u(x)$ be the ``maximal section contained in $\Omega_1$" centered at $x$, so that $S_{\bar{h}(x)}^u(x)$ is contained in $\Omega_1$ and tangent to $\partial \Omega_1$. The existence of such a section follows from the continuity of the sections in $h$, see \cite{CM}. By (\ref{DiamBound}) and the universal Lipschitz bound on $u$, there exists $c > 0$ universal such that for all $x$ in the $c$-neighborhood $\mathcal{N}_c$ of the union of the holes in $\Omega_1$, the section $S_{\bar{h}(x)}^u(x)$ is tangent to a hole and not $\partial \Omega_0$.

The arguments in \cite{SY} show that $u \in W^{2,\,p}(\Omega_1 \backslash \mathcal{N}_c) \cap C^{1,\,\alpha}(\overline{\Omega_1 \backslash \mathcal{N}_c})$ for any $p > 1$ and $\alpha \in (0,\,1)$, with corresponding estimates in these spaces depending on $\delta,\, p$ and $\alpha$. We may thus focus our attention on $\mathcal{N}_c$.

To that end, let $x \in \mathcal{N}_c$, and let $d$ be the distance from $x$ to the boundary. We will prove that
\begin{equation}\label{HessianBlowup}
|D^2u(x)| \leq c^{-1}d^{-1/2}.
\end{equation}
The $W^{2,\,p}$ estimate from Theorem \ref{Main} follows immediately, and the $C^{1,\,1/2}$ estimate comes from integrating (\ref{HessianBlowup}) along line segments.

Assume after a translation that $S_{\bar{h}(x)}^u(x)$ is tangent to a hole at the origin. Applying Proposition \ref{Engulfing} we engulf $S_{\bar{h}(x)}^u(x)$ by $R_h(0)$, with $h \sim \bar{h}(x)$. We will prove that
$$\eta_h(0) \leq c^{-1}d^{-1/2}.$$
Using that $|S_{\bar{h}(x)}^u(x)| \sim |R_h(0)|$ it is easy to see that $\eta_{\bar{h}(x)}(x) \leq c^{-1}\eta_h(0)$. Combining this with the above inequality and (\ref{HessianBound}) gives (\ref{HessianBlowup}). 

In what follows we will use that if $h_1 \sim h_2$ then $R_{h_1}(0)$ is approximated by $R_{h_2}(0)$ (see Appendix), whence $\lambda_{h_1}(0) \sim \lambda_{h_2}(0),\, \Lambda_{h_1}(0) \sim \Lambda_{h_2}(0),$ and $\eta_{h_1}(0) \sim \eta_{h_2}(0)$. We will also denote $\eta_h(0)$ by $\eta_h$, and we will similarly drop the notation $(0)$ from the other relevant quantities.

If either $\eta_h \leq c_2^{-1}M_2$ or the conditions of Lemma \ref{ModelGeometry} are satisfied, then we are done. This is obvious in the first case. In the second, Lemma \ref{ModelGeometry} gives $\Lambda_h^2 \leq c^{-1}\lambda_h$, hence $\eta_h \leq c^{-1}\lambda_h^{-1/2}$. Moreover, the distance from boundary in $R_h$ is at most $c^{-1}\lambda_h$. In particular, $\lambda_h^{-1/2} \leq c^{-1}d^{-1/2}$, and we are also done in this case.

So assume that neither is satisfied, and let $\hat{t}$ be the supremum of heights $t$ such that $\eta_s > c_2^{-1}M_2$ and the conditions of Lemma \ref{ModelGeometry} are not satisfied at height $s$ for all $s \in [h,\,t]$. 

We first assume that $\hat{t} \leq 1$. Then there exist $\hat{h} \in [h,\,\hat{t}]$ and $\kappa \in (1,\,2)$ such that $\kappa \hat{h} \geq \hat{t}$ and either $\eta_{\kappa\hat{h}} \leq c_2^{-1}M_2$ or the conditions of Lemma \ref{ModelGeometry} are satisfied at height $\kappa\hat{h}$, but neither holds at height $\hat{h}$. At height $\hat{h}$, we can either apply Lemma \ref{SmallFrac} or \ref{SmallFrac2}, giving a conclusion at height $M_1^{-1}\hat{h}$ or $M_2^{-1}\hat{h}$. Repeat applying Lemma \ref{SmallFrac} or Lemma \ref{SmallFrac2} (say the former $k$ times and the latter $l$ times) until the first time $M_1^{-k}M_2^{-l}\hat{h} < h$. Assume that eccentricity increased in the application of Lemma \ref{SmallFrac2} $l' \leq l$ times.
We then have
$$\eta_h \leq c^{-1}c_1^{-k}c_2^{-l'}\eta_{\hat{h}} \leq c^{-1}r^{-1/2}\eta_{\kappa \hat{h}},$$
where 
$$r := c_1^{-k}M_1^{-k/2}c_2^{-l'}M_2^{-l'/2} = c_1^{2k}c_2^{2l'} < 1.$$
By Lemmas \ref{SmallFrac} and \ref{SmallFrac2}, we also have that
$$R_h \subset c^{-1}rR_{\kappa\hat{h}}.$$

We now consider the case that at height $\kappa\hat{h}$, the conditions of Lemma \ref{ModelGeometry} are satisfied. The lemma implies that
$$\eta_{\kappa\hat{h}} \leq c^{-1}\lambda_{\kappa\hat{h}}^{-1/2},$$
thus
\begin{equation}\label{ModelGeometryBound}
\eta_h \leq c^{-1}r^{-1/2}\lambda_{\kappa\hat{h}}^{-1/2}.
\end{equation}
Assume after a rigid motion that the picture is oriented as in the proof of Lemma \ref{ModelGeometry}, so that $R_{\kappa\hat{h}}$ has vertical short axis. Then the boundary of $\Omega_1$ contains a portion that lies above $\{x_2 = sx_1 - c^{-1}x_1^2\}$,
where $0 \leq s \leq c^{-1}\lambda_{\kappa\hat{h}}/\Lambda_{\kappa\hat{h}}$ and $\Lambda_{\kappa\hat{h}}^2 \leq c^{-1}\lambda_{\kappa\hat{h}}$ (by the proof of Lemma \ref{ModelGeometry}). Recall that $R_h$ is contained in the $c^{-1}r$ times dilation of $R_{\kappa\hat{h}}$. Thus, in $R_h$, the distance from boundary is at most
$$c^{-1}(r\lambda_{\kappa\hat{h}} + sr\Lambda_{\kappa\hat{h}} + r^{2}\Lambda_{\kappa\hat{h}}^2).$$
Using the previous inequalities we see that the second and third terms are bounded by $r\lambda_{\kappa\hat{h}}$ and $r^{2}\lambda_{\kappa\hat{h}}$, respectively, giving a bound on the distance between boundary in $R_h$ of the size
$$d \leq c^{-1}r\lambda_{\kappa\hat{h}}.$$
Rearranging gives
$$\lambda_{\kappa\hat{h}}^{-1/2} \leq c^{-1} r^{1/2}d^{-1/2}.$$
Using this in (\ref{ModelGeometryBound}) gives
$$\eta_h \leq c^{-1}d^{-1/2},$$
and we are done with this case.

In the case that $\eta_{\kappa\hat{h}} \leq c_2^{-1}M_2$, we have
$$\eta_h \leq c^{-1}r^{-1/2}.$$
Furthermore, we have that $R_h$ is contained in a universal dilation of $B_r$ since $R_{\hat{h}}$ is contained in $B_{c^{-1}}$, thus in $R_h$ the distance from the boundary is at most $c^{-1}r$, hence
$$\eta_h \leq c^{-1}d^{-1/2}$$
in this case as well.

Finally, we deal with the case that $\hat{t} > 1$. Since $\eta_1$ is still bounded by a universal constant, we can take $\hat{h} = 1$ and repeat exactly the same arguments as above. More precisely, by repeated application of Lemma \ref{SmallFrac} or Lemma \ref{SmallFrac2} starting from height $\hat{h} = 1$, we get
$$\eta_h \leq c^{-1}r^{-1/2}$$ 
and
$$R_h \subset B_{c^{-1}r},$$
where $r$ is defined in the same way as above. Combining these two conclusions we get $\eta_h \leq c^{-1}d^{-1/2}$, and this completes the proof.
\end{proof}

%%%%%%%%%%%%%%%%%%
\section{Future Directions}\label{FutureDirections}

To conclude the paper we list a few questions to be investigated in future work.

\vspace{3mm}

\begin{enumerate}
\item Establish Sobolev regularity for optimal transport maps of a natural class of non-convex domains in higher dimensions. As noted above, the Pogorelov-type estimate Proposition \ref{Pogorelov2} is convenient but not required for the result in this paper, so there is hope for such an extension.

\vspace{3mm}

\item In two dimensions, enlarge the class of source domains being considered e.g. to arbitrary smooth domains. Our arguments handle concave parts of the boundary since we only use the equation and not the convexity of the target domain. However, convex parts of the boundary that lie inside the convex hull of the source domain may be tricky to handle, since at such points we do not have duality, which played an important role in the works of Caffarelli \cite{C4} and Savin-Yu \cite{SY}.

\vspace{3mm}

\item Investigate applications of our results to the existence theory for the semigeostrophic equations in cases where the source density is allowed to vanish. In previous works dealing with the case where the source density is bounded from below by a positive constant, $W^{2,\,1}$ estimates played a central role (\cite{ACDF1}, \cite{ACDF2}, \cite{F}).

\vspace{3mm}

\item Investigate applications of the ideas in this paper to the partial regularity theory of optimal transport maps when the domains are not convex. In this case, optimal transport maps can have singularities, and interesting results have been obtained about the size of the singular set (\cite{DF0}, \cite{FK}, \cite{F2}, \cite{GO}). However, the fine geometric measure-theoretic structure of the singular set is not well-understood, even in two dimensions with smooth domains and quadratic cost. In that case, a reasonable conjecture seems to be that the one-dimensional Hausdorff measure of the singular set is bounded.
\end{enumerate}

%%%%%%%%%%%%%%%%%%
\section{Appendix}\label{Appendix}

In this section we provide some of the details that we omitted for simplicity of presentation above. We start with a simple lemma that will be used in some of the subsequent proofs.

\begin{lem}\label{Segment1}
There is no convex function $w$ on $B_1 \subset \mathbb{R}^2$ that satisfies 
$$\det D^2w \geq \chi_{\{x_2 > 0\}}, \quad w|_{\{x_2 = 0\}} \text{ linear}.$$ 
\end{lem}
\begin{proof}
After subtracting an affine function we may assume that $w \geq 0$ and $w|_{\{x_2 = 0\}} = 0$. After subtracting a multiple of $x_2$ we may assume further that $w(0,\,t) = o(t)$ for $t > 0$. It follows that, for any $k > 0$, we can choose $h > 0$ small such that
$$R := [-1/2,\,1/2] \times [0,\,2kh] \subset \{w < h\}.$$
The quadratic polynomial $Q = 8hx_1^2 + 2k^{-2}h^{-1}(x_2 - kh)^2$ thus lies above $w$ on the boundary of $R$, and for $k$ large satisfies $\det D^2Q < 1$. The comparison principle implies that $Q > w \geq 0$ in $R$, contradicting that $Q$ vanishes at the center of $R$.
\end{proof}

We now prove the claim that functions in $F_{\delta}$ are not linear along line segments centered in the closure of the source domain, from which the universal bound on diameters of sections followed.

\begin{lem}\label{Segment2}
Let $u \in \mathcal{F}_{\delta}$. Then $u$ is not linear along any line segment centered at a point in $\overline{\Omega_1}$.
\end{lem}
\begin{proof}
By Lemma \ref{Segment1} (appropriately rescaled), $u$ is not linear along any segment in $\overline{\Omega_1}$. The only remaining possibility is that $u$ is linear along a segment that is tangent to $\partial \Omega_0$ at a single point that lies inside the segment. Since $u$ cannot be linear along the whole line containing this segment (this would imply $\det D^2u \equiv 0$ by standard convex analysis), we conclude that the agreement set between $u$ and a linear function has extremal points outside of $\Omega_0$. This contradicts that $\det D^2u = 0$ outside of $\Omega_0$.
\end{proof}

Finally, the engulfing property Proposition \ref{Engulfing} follows from the following pair of lemmas. Before proceeding we recall a standard renormalization procedure. For $u \in \mathcal{F}_{\delta},\, x \in \overline{\Omega_1}$ and $S_h^u(x)$ contained in $\Omega_1$ or intersecting one hole, let
$$u_h(y) = \frac{1}{h}u(A_hy) + L_h,$$
where $A_h$ is an affine transformation of determinant $h$ that takes a square centered at $0$ to $x + R_h(x)$, and $L_h$ is a linear function chosen so that $u_h = 0$ on $\partial S_1^{u_h}(0)$. Then $u_h$ solves $\det D^2u_h = \chi_{\Omega_h}$ in the normalized domain $S_1^{u_h}(0)$, where $\Omega_h^c$ is convex and $\overline{\Omega_h}$ contains $0$. We call this procedure ``renormalizing in the section $S_h^u(x)$".

\begin{lem}\label{Engulfing0}
If $u \in \mathcal{F}_{\delta},\, y \in \overline{S_h^u(x)} \cap (\partial\Omega_1 \backslash \partial \Omega_0)$ and $S_h^u(x) \subset \Omega_1$, then 
$$S_h^u(x) \subset S_{c^{-1}h}^u(y).$$
\end{lem}
\begin{proof}
Assume by way of contradiction that the lemma is false. Then there exists a sequence $u_k \in \mathcal{F}_{\delta}$ such that $S_{h_k}^{u_k}(x_k)$ are contained in the source domains and $y_k$ are in their closures and the boundary of a hole (say $y_k = 0$ after a translation), but $S_{kh_k}^{u_k}(0)$ do not contain $S_{h_k}^{u_k}(x_k)$. Note that we may assume that $h_k \leq c^{-1}k^{-1}$ by the uniform Lipschitz bound on $u_k$, so $h_k \rightarrow 0$ and the sections localize close to the holes. After renormalizing in the sections $S_{kh_k}^{u_k}(0)$, we get a subsequence of rescalings of $u_k$ that converge to a function $w$ which satisfies that $S_1^w(0)$ is normalized, $\det D^2w = 1$ in a domain $\Omega$ with $\Omega^c$ convex, and $w$ is linear along a segment from $0$ to $\partial S_1^w(0)$ contained in $\overline{\Omega}$. This contradicts Lemma \ref{Segment1}, appropriately rescaled.
\end{proof}

\begin{lem}\label{Engulfing1}
Let $u \in \mathcal{F}_{\delta}$. Then for all $\alpha \in (0,\,1)$, there exists $\beta(\delta,\,\alpha) > 0$ such that 
$$S_{\beta h}^u(y) \subset \alpha S_h^u(y)$$
for all $y \in \partial \Omega_1 \backslash \partial \Omega_0$, where $\alpha S_h^u(y)$ is the $\alpha$-times dilation of $S_h^u(y)$ around $y$.
\end{lem}
\begin{proof}
The argument is similar to the one above. If the lemma is false, there is a sequence $u_k \in \mathcal{F}_{\delta}$ such that (up to translations) $0$ is in the boundary of a hole and $S_{h_k/k}^{u_k}(0)$ are not contained in $\alpha S_{h_k}^{u_k}(0)$. After renormalizing in the sections $S_{h_k}^{u_k}(0)$, we get a subsequence of rescalings that converge to a function $w$ which satisfies that $S_1^w(0)$ is normalized, $\det D^2w = 1$ in a domain $\Omega$ with $\Omega^c$ convex and $0 \in \partial \Omega$, and $w$ is linear along a line segment passing through the origin with an endpoint on $\partial(\alpha S_1^w(0))$. Again, we contradict Lemma \ref{Segment1}.
\end{proof}

Finally, we sketch the proof that if $u \in \mathcal{F}_{\delta},\, 0 \in \partial \Omega_1 \backslash \partial \Omega_0$, and $h_1 \sim h_2$, then $R_{h_1}(0)$ is approximated by (contains and is contained in dilations by universal constants of) $R_{h_2}(0)$. This fact was used in the proof of Theorem \ref{Main} in Section \ref{ProofMain}. It suffices to show that $S_{h_1}^u(0)$ is approximated by $S_{h_2}^u(0)$. After renormalizing in $S_{h_1}^u(0)$ we get a convex function $w$ such that $S_1^w(0) = \{w < 0\}$ is normalized, and we need to show that $S_c^w(0)$ is approximated by $B_1$, where $c \sim 1$. Since 
$$|S_c^w(0)| \sim 1$$ 
we just need to show that $S_c^w(0)$ contains a ball centered at $0$ that has small universal radius. If not, then by the local universal Lipschitz bound on $w$, the slope of the linear function defining $S_c^w(0)$ is extremely large, say after a rotation, $K e_1$ with $K >> 1$. But in this case the line segment in $S_c^w(0)$ through the origin parallel to $e_1$ would intersect $\partial S_c^w(0)$ at a distance much shorter from the origin on the left than on the right, contradicting that $0$ is the center of mass of $S_c^w(0)$ and completing the proof.

%%%%%%%%%%%%%%%%%%%%%%%%%%%%%%%%%%%%%%%%%%%%%%%%%%%%%%%%%%%%%%%%%%%%%%%%%%%%%%%%%%%

%%%%%%%%%%%%%%%%%%%%%%%%%%%%%%%%%%%%%%%%%%%%%%%%%%%%%%%%%%%%%%%%%%%%%%%%%%%%%%%%%%%%

\end{document}